\documentclass[preprint]{elsarticle}
\usepackage{amsfonts,amssymb,amsmath}

\newenvironment{proof}{\medskip                    
\noindent{\scshape Proof:}}{\quad $\square$
\medskip}  

\usepackage{hyperref}

\newtheorem{theorem}{Theorem}[section]
\newtheorem{lemma}{Lemma}[section]
\newtheorem{proposition}{Proposition}[section]
\newtheorem{corollary}{Corollary}[section]

\usepackage{tikz}

\newcommand {\cC} {{\cal C}}

\newcommand {\cG} {{\cal G}}

\newcommand {\cK} {{\cal K}}

\newcommand {\cT} {{\cal T}}

\newcommand{\cycles}{\cC}
\newcommand{\digr}{\cG}

\newcommand {\R} {{\mathbb R}}

\newcommand {\Rp} {\R_+}
\newcommand {\Rpn} {\R_+^n}
\newcommand {\Rpnn} {\Rp^{n\times n}}

\newcommand {\taccess}{\longrightarrow_{\tau}}
\newcommand {\taccesses}{\longrightarrow_{\tau}}
\newcommand {\nottaccess}{\not\longrightarrow_{\tau}}
\newcommand {\spann}{\operatorname{span}_{\oplus}}
\newcommand{\supp}{\operatorname{dom}}
\newcommand{\suppp}{\operatorname{supp}}
\newcommand{\cTadm}{\cT_{\operatorname{adm}}}
\newcommand{\cTag}{\cT_{\operatorname{ag}}}
\newcommand{\bunity}{1}
\newcommand{\bzero}{0}

\begin{document}

\title{Extremals of the supereigenvector cone in max algebra: a combinatorial description}

\author[rvt]{Serge{\u{\i}} Sergeev }
\ead{sergiej@gmail.com}

\address[rvt]{University of Birmingham, School of Mathematics, Edgbaston B15 2TT, UK}

\fntext[fn1]{Supported by EPSRC grant EP/J00829X/1, RFBR grant
12-01-00886.}

\begin{abstract}
We give a combinatorial description of extremal generators of the supereigenvector cone
$\{x\colon A\otimes x\geq x\}$ in max algebra.
\end{abstract}

\begin{keyword}
Eigenvectors, eigenvalues, max-plus, max algebra, digraph \vskip0.1cm {\it{AMS
Classification:}} 15A80, 15A06, 15A18
\end{keyword}

\maketitle

\section{Introduction}

By max algebra we understand the semiring of nonnegative numbers $\Rp$ equipped with arithmetical
operations of ``tropical addition'' $a\oplus b=\max(a,b)$ (instead of the usual one), and the
ordinary multiplication. See~ Butkovi\v{c}~\cite{But} for one of the recent textbooks, as well as
Heidergott, Olsder and van der Woude~\cite{HOW}
for another textbook explaining a typical application of max algebra to scheduling problems.
These arithmetical operations are extended to matrices and vectors in the usual way: for two matrices
$A$ and $B$ of appropriate sizes, we have
$(A\oplus B)_{ij}=a_{ij}\oplus b_{ij}$ and $(A\otimes B)_{ik}=\bigoplus_j a_{ij}b_{jk}$.
We also consider the max-algebraic powers of matrices: $A^{\otimes t}=\underbrace{A\otimes\ldots\otimes A}_t$.

With each square matrix $A\in\Rpnn$ we can associate a weighted directed digraph 
$\digr(A)=(N,E)$ with set of nodes $N=\{1,\ldots,n\}$ and edges $E=\{(i,j)\mid a_{ij}\neq 0\}$. 
Each matrix entry $a_{ij}$ is the weight of edge $(i,j)$. 

A sequence of edges 
$(i_1,i_2),\ldots,(i_{k-1},i_k)$ of $\digr(A)$ is called a {\em walk}. The {\em length} of this walk
is $k-1$,  and the {\em weight}
of this walk is defined as $a_{i_1i_2}\cdot\ldots\cdot a_{i_{k-1}i_k}$. Node $i_1$ is called the 
{\em beginning node}, and $i_k$ is called the {\em final node} of that walk. If $i_1=i_k$ then the
walk is called a {\em cycle}.

It is easy to see that the $i,j$ entry of the
max-algebraic power $A^{\otimes t}$ is equal to the greatest weight of a walk of length 
$t$ beginning at $i$ and ending at $j$.  Considering the formal series
\begin{equation}
\label{e:kls}
A^*=I\oplus A\oplus A^{\otimes 2}\oplus\ldots\oplus A^{\otimes k}\oplus\ldots,
\end{equation}
called the {\em Kleene star} of $A$ we see that the $i,j$ entry of $A^*$ is equal to the greatest weight
among all walks connecting $i$ to $j$ with no restriction on weight. This greatest weight is
defined for all $i,j$ if and only if $\digr(A)$ does not have cycles with weight exceeding $1$, 
otherwise~\eqref{e:kls} diverges, or more precisely, 
some entries of $A^*$ diverge to $+\infty$.

In this paper we consider the problem of describing the set of supereigenvectors of a given square matrix
$A\in\Rpnn$.  These are vectors $x$ satisfying $A\otimes x\geq x$, so we are interested in the set
\begin{equation}
\label{e:v*a}
V^*(A)=\{x\colon A\otimes x\geq x\}.
\end{equation} 
Supereigenvectors  are of interest for several reasons.  Let us first mention that the problem which is solved in this paper
was posed by  Butkovi\v{c}, Schneider and Sergeev~\cite{BSSws}, where the supereigenvectors
were shown to be instrumental in the analysis of the sequences $\{A^k\otimes x\colon k\geq 1\}$. A partial solution to that problem has been described by Wang and Wang~\cite{Wang}.

Furthermore, the set
of max-algebraic eigenvectors of $A$ (here, associated with eigenvalue $1$)  
and the set of subeigenvectors of $A$ defined, respectively, as 
\begin{equation}
V(A)=\{x\colon A\otimes x=x\},\quad V_*(A)=\{x\colon A\otimes x\leq x\},
\end{equation}
have been well studied and thoroughly described in the literature. 
Let us also mention that $A\otimes x\geq x$ belongs to the class of two-sided systems $A\otimes x\leq B\otimes x$, whose 
polynomial solvability is still under question, while it is known that the problem is in the intersection of
NP and co-NP classes, see for instance Bezem, Nieuwenhuis and Rodriguez-Carbonell~\cite{RC+}. A number of algorithms solving this general problem
and describing the full solution set have been 
designed: see, in particular, the double description method of~ Allamigeon, Gaubert and Goubault~\cite{AGG}. 


$V(A)$, $V^*(A)$ and $V_*(A)$ are examples of max cones. Recall that a subset of $\Rpn$ is
called a max cone if it is closed under addition $\oplus$ of its elements, and under the usual
scalar multiplication.  The description that we seek is in terms of max-algebraic generating sets 
and bases. Let us recall some definitions that are necessary here.

An element $u\in\Rpn$ is called a {\em max combination} of elements $v^1,\ldots, v^m\in \Rpn$
if there exist scalars $\lambda_1,\ldots,\lambda_m\in\Rp$ such that 
$u=\bigoplus_{i=1}^m \lambda_i v^i$.  Further $S\subseteq \Rpn$ is called a {\em generating set} for a max cone $\cK$ if every element of $\cK$ can be represented as a max combination of some elements of $S$.
If $S$ is a generating set of $\cK$, we write $\cK=\spann(S)$. 
Further, $S$ is called
a {\em basis} if none of the elements of $S$ is a max combination of other elements of $S$.

An element $u$ of a max cone $\cK\subseteq\Rpn$ is called an {\em extremal}, if whenever $u=v\oplus w$  and $v,w\in\cK$, we have $u=v$ or $u=w$. An element $u\in\Rpn$ is called {\em scaled} if $\max_{i=1}^n u_i=1$.
A basis of a max cone is called {\em scaled} if so is every element of that basis.


\begin{proposition}[{\cite{BSS},\cite{GK-07}}]
\label{p:base-extr}
For any closed max cone $\cK\subseteq\Rpn$, let $E$ be the set of scaled 
extremals.  Then $E$ is non-empty, $\cK=\spann(E)$ 
and, furthermore, $E$ is a unique scaled basis of $\cK$.
\end{proposition}
 
 
It is easy to see that the max cones $V^*(A)$, $V_*(A)$ and $V(A)$ are closed, so that 
Proposition~\ref{p:base-extr} applies to them. In fact, all these cones
have a finite number of scaled extremals, which constitute their essentially unique bases.
Our purpose will be to describe the generating set of the supereigenvector cone $V^*(A)$ and then to single out those
generators that are extremals and thus form a basis of $V^*(A)$.

The rest of the paper is organized as follows. In Section 2 we describe a generating set of the supereigenvector cone. This description, obtained in Theorem~\ref{t:kitai}, is equivalent to the result of Wang
and Wang~\cite{Wang}, but it is obtained using a more geometric ``cellular decomposition''
technique. 
In Section 3 we give criteria under which the generators described in Section 2 are extremals.  This description
and these criteria are combinatorial in nature, and expressed in terms of certain cycles of the digraph associated 
with the matrix (namely, cycles whose weight is not less than $1$). This description is the main result of the paper,
formulated in Theorems~\ref{t:criter},~\ref{t:criter-var} and~\ref{t:unitycase}. These results, to the author's knowledge, are new.

\section{Generating sets}

Let $A\in\Rpnn$.
A mapping $\tau$ of a subset of $[n]$ into itself will be called a
{\em (partial) strategy} of $\digr(A)$. Given a strategy $\tau$ we
can define the matrix $A^{\tau}=(a^{\tau}_{i,j})$ by
\begin{equation}
\label{e:ataudef}
a^{\tau}_{i,j}=
\begin{cases}
a_{i,j}, & \text{if $j=\tau(i)$},\\
0, & \text{otherwise}.
\end{cases}
\end{equation}
By domain of $\tau$, denoted by $\supp(\tau)$, we mean the set of indices $i$ for which $\tau(i)$ is
defined, that is, the index subset which $\tau$ maps into itself. 

If $\tau$ is a strategy then its inverse,
denoted by $\tau-$, is, in general, a multivalued mapping of a subset of
$\supp(\tau)$ to the whole $\supp(\tau)$ . Define the matrix $A^{\tau-}=(a^{\tau-}_{i,j})$
by
\begin{equation}
\label{e:atau-1def}
a^{\tau-}_{i,j}=
\begin{cases}
a^{-1}_{j,i}, & \text{if $i=\tau(j)$},\\
0, & \text{otherwise}.
\end{cases}
\end{equation}

Consider the associated digraphs $\digr(A^{\tau})$ and
$\digr(A^{\tau-})$ (see Figure~\ref{f:digr}.
Let us list some properties of $\digr(A^{\tau})$.

\begin{figure}
\begin{tabular}{cccc}

\begin{tikzpicture}[shorten >=1pt,->,scale=1]

\tikzstyle{vertex}=[circle,draw]

\foreach \name/\angle/\text in {1/0/1, 2/300/2,3/240/3, 4/180/4, 5/120/5, 6/60/6}
\node[vertex,xshift=0cm,yshift=0cm] (\name) at (\angle:1.7cm) {$\text$};

\draw node[vertex,xshift=-1.5cm,yshift=-3.0cm] (7)  {$7$};

\draw node[vertex,xshift=-2cm,yshift=-4.3cm] (8)  {$8$};

\draw node[vertex,xshift=-1cm,yshift=-4.3cm] (9)  {$9$};


\draw node[vertex,xshift=1.5cm,yshift=-3.0cm] (11)  {$11$};

\draw node[vertex,xshift=2cm,yshift=-4.3cm] (12)  {$12$};


\draw [->] (1) to node[right] {$a_{1,2}$} (2);

\draw [->] (2) to node[below] {$a_{2,3}$}(3);

\draw [->] (3) to node[left] {$a_{3,4}$} (4);
\draw [->] (4) to node[left] {$a_{4,5}$} (5);
\draw [->] (5) to node[above] {$a_{5,6}$} (6);
\draw [->] (6) to node[right] {$a_{6,1}$} (1);


\draw [->] (8) to node[left] {$a_{8,7}$}   (7);

\draw [->] (9) to node[right] {$a_{9,7}$}   (7);

\draw [->] (7) to node[left] {$a_{7,3}$}  (3);

\draw [->] (12) to node[right] {$a_{12,11}$}(11);
\draw [->] (11) to node[right] {$a_{11,2}$} (2);


\end{tikzpicture}
&&&

\begin{tikzpicture}[shorten >=1pt,->,scale=1]

\tikzstyle{vertex}=[circle,draw]

\foreach \name/\angle/\text in {1/0/1, 2/300/2,3/240/3, 4/180/4, 5/120/5, 6/60/6}
\node[vertex,xshift=0cm,yshift=0cm] (\name) at (\angle:1.7cm) {$\text$};

\draw node[vertex,xshift=-1.5cm,yshift=-3.0cm] (7)  {$7$};

\draw node[vertex,xshift=-2cm,yshift=-4.3cm] (8)  {$8$};

\draw node[vertex,xshift=-1cm,yshift=-4.3cm] (9)  {$9$};


\draw node[vertex,xshift=1.5cm,yshift=-3.0cm] (11)  {$11$};

\draw node[vertex,xshift=2cm,yshift=-4.3cm] (12)  {$12$};

\draw [->] (2) to node[left] {$a_{1,2}^{-1}$} (1);

\draw [->] (3) to node[below] {$a_{2,3}^{-1}$} (2);

\draw [->] (4) to node[left] {$a_{3,4}^{-1}$} (3);
\draw [->] (5) to node[left] {$a_{4,5}^{-1}$}(4);
\draw [->] (6) to node[above] {$a_{5,6}^{-1}$} (5);
\draw [->] (1) to node[right] {$a_{6,1}^{-1}$} (6);


\draw [->] (7) to node[left] {$a_{8,7}^{-1}$} (8);

\draw [->] (7) to node[right] {$a_{9,7}^{-1}$} (9);

\draw [->] (3) to node[left] {$a_{7,3}^{-1}$}(7);

\draw [->] (11) to node[right] {$a_{11,2}^{-1}$} (12);
\draw [->] (2) to node[right] {$a_{12,11}^{-1}$} (11);


\end{tikzpicture}

\end{tabular}

\caption{Digraphs $\digr(A^{\tau})$ and $\digr(A^{\tau-})$ 
for $A\in\R_+^{11}$ (an example)  \label{f:digr}}
\end{figure}

\begin{lemma}
\label{l:digratau}
\begin{itemize}
\item[{\rm (i)}] For every pair of nodes of $[n]$, either there is a unique walk in
$\digr(A^{\tau})$ connecting one of these nodes to the other, or there is no such
walk.
\item[{\rm(ii)}] $\digr(A^{\tau})$ contains at least one cycle.
\item[{\rm (iii)}] For each node of $\supp(\tau)$, there is a unique
cycle of $\digr(A^{\tau})$ that can be accessed from this node via a walk in
$\digr(A^{\tau})$, which is also unique.
\item[{\rm (iv)}] For each node of $\supp(\tau)$, there are no nodes
that can be accessed from it by a walk of $\digr(A^{\tau})$  other than
the nodes of the unique cycle and the unique access walk mentioned in
(iii).
\end{itemize}
\end{lemma}

\if{
Let us also write some very similar properties of $\digr(A^{\tau-})$.
\begin{lemma}
\label{l:digratau-1}
\begin{itemize}
\item[{\rm (i)}] For every pair of nodes of $[n]$, either there is a unique walk in
$\digr(A^{\tau-})$ connecting one of these nodes
to the other, or there is no such walk.
\item[{\rm(ii)}] $\digr(A^{\tau-})$ contains at least one cycle.
\item[{\rm (iii)}] For each node of $\supp(\tau)$, there is a unique
cycle of $\digr(A^{\tau-})$ that accesses  this node via a walk in
$\digr(A^{\tau-})$, which is also unique.
\item[{\rm (iv)}] For each node of $\supp(\tau)$, there are no nodes
that access it via a walk of $\digr(A^{\tau-})$ other than
the nodes of the unique cycle and the unique access walk mentioned in
(iii).
\end{itemize}
\end{lemma}
}\fi

A strategy $\tau$ is called {\em admissible} if there is no cycle in
$\digr(A^{\tau})$ whose weight is smaller than $\bunity$. In this case, there is no cycle of
$\digr(A^{\tau-})$ whose weight is greater than $\bunity$, hence we have
$\lambda(A^{\tau-})\leq \bunity$.

The set of all admissible strategies is denoted by
$\cTadm(A)$. Let us argue that the set of all supereigenvectors can be represented as
union of the sets of subeigenvectors of $A^{\tau-}$ with
$\tau$ ranging over all admissible strategies.

\begin{proposition}
\label{p:repr}
\begin{equation}
\label{e:repr}
V^*(A)=\bigcup_{\tau\in\cTadm(A)} V^*(A^{\tau})
= \bigcup_{\tau\in\cTadm(A)} V_*(A^{\tau-})
\end{equation}
\end{proposition}
\begin{proof}
To prove that
\begin{equation}
\label{e:repr1}
V^*(A)=\bigcup_{\tau\in\cTadm(A)} V^*(A^{\tau}),
\end{equation}
 observe that
every vector $x$ satisfying $A\otimes x\geq x$ also satisfies
$A^{\tau}\otimes x\geq x$ for some (partial) mapping $\tau$ which
can be defined as follows:
\begin{equation}
\supp(\tau)=\{i\colon x_i\neq \bzero\},\quad \tau(i)\colon\  a_{i,\tau(i)} x_{\tau(i)}
=\max_j a_{i,j} x_j .
\end{equation}
The choice of $\tau(i)$ among the indices attaining maximum is free,
any such index can be taken for $\tau(i)$.

It can be verified that if $x_i>\bzero$ then $a_{i,\tau(i)} x_{\tau(i)}>\bzero$,
hence $x_{\tau(i)}>0$, thus $\tau$ maps $\supp(\tau)$ into itself,
so it is a strategy. To check that it is admissible let $i,\tau(i),\ldots,\tau^{\ell}(i)=i$
constitute a cycle, so we have $a_{i,\tau(i)}x_{\tau(i)}\geq x_i$, ...,
$a_{\tau^{\ell-1}(i),i} x_i \geq x_{\tau^{\ell-1}(i)}$. Multiplying up all these inequalities
and cancelling the product of $x_i$'s we get that the cycle weight is not less than
$\bunity$. This shows that $\tau$ is admissible. To complete the proof
of~\eqref{e:repr1} observe that $A\otimes x\geq A^{\tau}\otimes x\geq x$ for
every mapping $\tau$ and every vector $x$ satisfying
$A^{\tau}\otimes x\geq x$.

It remains to check that  $V^*(A^{\tau}) =V_*(A^{\tau-})$ for every partial
mapping $\tau$. We have
\begin{equation}
\label{e:v^*v_*}
\begin{split}
V^*(A^{\tau})=& \{x\colon a_{i,\tau(i)}x_{\tau(i)}\geq x_i\;\forall i\in\supp(\tau)\}=\\
&\{x\colon (a_{i,\tau(i)})^{-1}x_i\leq  x_{\tau(i)}  \;\forall i\in\supp(\tau)\}
=V_*(A^{\tau-}).
\end{split}
\end{equation}
Combined with~\eqref{e:repr1}, this implies~\eqref{e:repr}.
\end{proof}

Thus the cones $V^*(A^{\tau})=V_*(A^{\tau-})$, with $\tau$ ranging over
all admissible strategies, can be considered as building blocks of
$V^*(A)$. Hence the generating set of $V^*(A)$  can be formed as
the union of all generating sets of $V^*(A^{\tau})=V_*(A^{\tau-})$:
these are the generating sets of subeigenvector cones.
A generating set for a general subeigenvector cone $V_*(A)$ is easy to find.

\begin{proposition}[{e.g.~\cite{BCOQ},~\cite{But},~\cite{Ser-09}}]
\label{p:subeig-kleene}
Let $A\in\Rp^{n\times n}$ be such that the weight of any cycles of $\digr(A)$ does not exceed one.
Then $V_*(A)=\spann(A^*)$.
\end{proposition}

We now specialize this description to $V_*(A^{\tau-})$. For this purpose, let us denote by $k\taccess i$ the situation
when $k=i$ or $k$ can be connected to $i$ by a walk on $\digr(A^{\tau})$.
In this case, the unique walk
connecting $k$ to $i$ on $\digr(A^{\tau})$ will be denoted by
$P^{\tau}_{ki}$.

\begin{proposition}
\label{p:subtau}
Let $\tau$ be an admissible strategy. Then $V^*(A^{\tau})=V_*(A^{\tau-})$
is generated by the vectors $x^{(\tau,k)}$ for $k=1,\ldots,n$, whose coordinates
are defined as follows:
\begin{equation}
\label{e:xkdef}
x^{(\tau, k)}_i=
\begin{cases}
\bunity, & \text{if $i=k$},\\
(w(P^{\tau}_{ki}))^{-1}, & \text{if $i\neq k$ and $k\taccess i$},\\
\bzero, & \text{if $i\neq k$ and $k\nottaccess i$}.
\end{cases}
\end{equation}
\end{proposition}
\begin{proof}
By Proposition~\ref{p:subeig-kleene}, $V_*(A^{\tau-})=
\spann((A^{\tau-})^*)$, so it amounts to argue that the columns of
$(A^{\tau-})^*$ are exactly $x^{(\tau,1)},\ldots, x^{(\tau, n)}$. This claim follows
by the optimal walk interpretation of the entries $\alpha_{i,k}$ of
$(A^{\tau-})^*$: we obtain that $\alpha_{i,k}=x^{(\tau,k)}_i$ as defined in~\eqref{e:xkdef},
Indeed, recalling that $\alpha_{kk}=\bunity$ for all $k$,
 $i$ accesses $k$ in $\digr(A^{\tau-})$ if and only if $k$ accesses $i$
in $\digr(A^{\tau})$ and that
the weight of the unique access walk from $i$ to $k$ in $\digr(A^{\tau-})$ is the
reciprocal of the weight of the unique access walk from $k$ to $i$
in $\digr(A^{\tau})$, we obtain the claim from the optimal walk interpretation
of the entries of the Kleene star.
\end{proof}

Denote by $\cycles^{\geq 1}(A)$, respectively by $\cycles^{>1}(A)$,  the set of cycles
in $\digr(A)$ whose weight is not less than $\bunity$,
respectively greater by $1$. 

If $\digr(A^{\tau})$, for a strategy $\tau$, consists of one cycle and
one non-empty walk connecting its origin to a node of that cycle, then $\tau$ is called a
{\em germ}. The origin of that walk will be denoted
by $o_{\tau}$. If the weight of the cycle is no less than $1$ then the germ is called {\em admissible}.
The set of all admissible
germs in $\digr(A)$ will be denoted by $\cTag(A)$.

Obviously, both $\cycles^{\geq 1}(A)\subseteq\cTadm(A)$ and $\cTag(A)\subseteq\cTadm(A)$.
The following theorem describes
a generating set of $V^*(A)$ by means of nonnegative cycles and admissible germs.

\begin{theorem}
\label{t:kitai}
We have $V^*(A)=\spann(S)$ where
\begin{equation}
\label{e:kitai}
S=\{x^{(\tau,o_{\tau})}\colon \tau\in\cTag(A)\}\cup\{x^{(\tau,k)}\colon\tau\in\cycles^{\geq 1}(A),\,k\in \supp(\tau)\}.
\end{equation}
\end{theorem}
\begin{proof}
Since every admissible germ and every nonnegative cycle is an admissible strategy, inclusion
$V_*(A^{\tau-})\subseteq V^*(A)$ and Proposition~\ref{p:subtau}
imply that
$$\spann\{x^{(\tau,o_{\tau})}\colon \tau\in\cTag(A)\}\subseteq V^*(A)$$
and 
$$\spann\{x^{(\tau,k)}\colon\tau\in\cycles^{\geq 1}(A),\,k\in \supp(\tau)\}\subseteq V^*(A)$$.

To prove the opposite inclusion, observe that for each 
$x^{(\tau,k)}$ of~\eqref{e:xkdef} we can define a new strategy $\tau'$ by
\begin{equation}
\supp(\tau')=\{i\colon k\taccess i\},\quad
\tau'(l)=\tau(l)\ \text{if $l\in\supp(\tau')$},
\end{equation}
and then we have $x^{(\tau',k)}=x^{(\tau,k)}$. We now argue that $\tau'$ is a more simple strategy than
$\tau$. 

By Lemma~\ref{l:digratau} part (iv), in $\digr(A^{\tau})$ node $k$ only
accesses one nonnegative cycle and the nodes on the unique walk leading to that cycle.
It follows that either $\tau'\in\cycles^{\geq 1}$ and then $k\in \supp(\tau')$,
or $\tau'\in\cTadm(A)$ and $k=o_{\tau'}$
Hence for every 
generator 
of $V^*(A)$, expressed as  $x^{(\tau,k)}$ , there exists $\tau'$ which is either a nonnegative
cycle with one of the nodes being $k$, or it is an admissible germ and $k=o_{\tau'}$, in any case such that
$x^{(\tau',k)}=x^{(\tau,k)}$. 
This implies that
$$
V^*(A)\subseteq \spann\{x^{(\tau,o_{\tau})}\colon \tau\in\cTag(A)\}\oplus
\spann\{x^{(\tau,k)}\colon\tau\in\cycles^{\geq 1}(A),\,k\in\supp(\tau)\}$$. The theorem
is proved.`
\end{proof}

\section{Extremals}

Let us introduce the following partial order relation.

\begin{equation}
\label{e:leqi}
y\leq_i x\;\text{if $x_i\neq 0$, $y_i\neq 0$ and}\; y_ky_i^{-1}\leq x_kx_i^{-1}\;\forall k.
\end{equation}

In particular, this relation is transitive:

\begin{equation}
x\leq_i y\leq_i z \Rightarrow x\leq_i z.
\end{equation}

The following fact is known, see~\cite[Proposition 3.3.6]{But} or~\cite[Theorem 14]{BSS}, and
also~\cite{Ser-09}.

\begin{proposition}
\label{p:multiorder}
Let $S\subseteq\Rpn$ and $\cK=\spann(S)$. Then $x$ is not an extremal of
$\cK$ if and only if for each $i\in\suppp(x)$ there exists $y^i\in S$ such that
$y^i\leq_i x$ and $y^i\neq x$.
\end{proposition}

We consider the case when $\cK=V^*(A)$. A generating set of this max cone
is given in Theorem~\ref{t:kitai}. Our purpose is to identify extremals, which yield
an essentially unique basis of $V^*(A)$, by means of the criterion
described in Proposition~\ref{p:multiorder}.


We first show that for each $\tau$ and $k$, there is a relation between
$x^{(\tau,k)}$ and $x^{(\tau,\tau(k))}$, with respect to every preorder relation except for $\leq_k$.

\begin{lemma}
\label{l:xtau}
Let $\tau\in\cTag(A)$ and $k=o_{\tau}$ or $\tau\in\cycles^{\geq 1}(A)$ and $k\in \supp(\tau)$.
\begin{itemize}
\item[{\rm (i)}]   $x^{(\tau,\tau(k))}\leq_j x^{(\tau,k)}$ for all $j\in\supp(\tau)$ and $j\neq k$.
\item[{\rm (ii)}]  $x^{(\tau,\tau(k))}\neq x^{(\tau,k)}$ if and only if $\tau\in\cTag(A)$ or
$\tau\in\cycles^{>1}(A)$.
\end{itemize}
\end{lemma}
\begin{proof}
Let $\tau\in\cTag(A)$ and $k=o_{\tau}$. Then $\suppp(x^{(\tau,\tau(k))})=\supp(\tau)
\backslash\{k\}$, and in particular,
$x^{(\tau,\tau(k))}\neq x^{(\tau,k)}$. As we also have 
$ (x_j^{(\tau,\tau(k))})^{-1} x_i^{(\tau,\tau(k))}=
 (x_j^{(\tau,k)})^{-1} x_i^{(\tau,k)}=w(P_{ij}^{\tau})$ 
for all $i,j\neq k$, claim (i) follows, in the case when $\tau\in\cTag(A)$.

Let $\tau\in\cycles^{\geq 1}(A)$. Then $\suppp(x^{(\tau,\tau(k))})= \suppp(x^{(\tau,k)})=\supp(\tau)$, and
$$\left(x_j^{(\tau,\tau(k))}\right)^{-1}x_i^{(\tau,\tau(k))}= \left(x_j^{(\tau,k)}\right)^{-1}
x_i^{(\tau,k)}=w(P^{\tau}_{ij})\  \text{for}\ i,j\neq k.$$
However, we also
have
\begin{equation}
\label{xktau-ineq}
\left(x_i^{(\tau,\tau(k))}\right)^{-1}\cdot  x_k^{(\tau,\tau(k))}\leq 
\left(x_i^{(\tau,k)}\right)^{-1}\cdot x_k^{(\tau,k)}\ \text{for}\ i\neq k,
\end{equation}
since
\begin{equation}
\label{xktau-expl}
\begin{split}
&x_k^{(\tau,k)}\left(x_i^{(\tau,k)}\right)^{-1}= w(P^{\tau}_{ki}),\quad
x_k^{(\tau,\tau(k))}\left(x_i^{(\tau,\tau(k))}\right)^{-1}=(w(P^{\tau}_{ik}))^{-1},\\
&w(P^{\tau}_{ki})w(P^{\tau}_{ik})=w(\tau)\geq\bunity.
\end{split}
\end{equation}
Furthermore, we have $x^{(\tau,\tau(k))}\neq x^{(\tau,k)}$
if and only if the inequality in~\eqref{xktau-ineq} is strict, which happens if and only if $w(\tau)>\bunity$.
Hence both claims.
\end{proof}

If $\tau\in\cTag(A)$ and $k=o_{\tau}$ or $\tau\in\cycles^{\geq 1}(A)$ and $k\in \supp(\tau)$, there is a unique walk
issuing from $k$ and containing all nodes of $\supp(\tau)$. Denote the final node of that
walk by $\operatorname{endn}(\tau,k)$.

\begin{corollary}
\label{c:xtau}
Let $\tau\in\cTag(A)$ and $k=o_{\tau}$ or 
$\tau\in\cycles^{\geq 1}(A)$ and $k\in \supp(\tau)$. Let $i\neq k$
be any index in $\supp(\tau)$. Then $x^{(\tau,i)}\leq_i x^{(\tau,k)}$.
\end{corollary}
\begin{proof}
Without loss of generality we will assume that
the nodes of $\tau$, where $\tau$ is a cycle or a germ, are numbered
in such a way that $k=1$ and $\tau(i)=i+1$ for all $i\in\supp(\tau)$ except
for the node $\operatorname{endn}(\tau,k)$ which has the greatest number
$m$. Note that if $\tau$ is a cycle then $\tau(m)=1$.

Repeatedly applying Lemma~\ref{l:xtau} part (i), we have
\begin{equation}
x^{(\tau,i)}\leq_i x^{(\tau,i-1)}\leq_i x^{(\tau,i-2)}\leq_i\ldots\leq_i x^{(\tau,1)}.
\end{equation}
\end{proof}

We now formulate and prove the main results of the paper, which constitute a combinatorial 
characterization of the supereigenvector cone $V^*(A)$.
Let us distinguish between germs whose unique cycle has weight 
strictly greater than $1$, whose set we denote by $\cTag^{>1}(A)$, and the set of germs whose unique cycle 
has weight $1$, whose set we denote by $\cTag^{=1}(A)$.

\begin{theorem}
\label{t:criter}
Let $\tau\in\cTag^{>1}(A)$ and $k=o(\tau)$ or $\tau\in\cycles^{>1}(A)$ and $k\in \supp(\tau)$.
Then $x^{(\tau,k)}$ is not an extremal if and only if one of the following conditions hold:
\begin{itemize}
\item[{\rm (i)}] there exist $i,l$ and $j$ such that $i\neq l\neq j$, $i\taccesses l\taccesses j$
and $a_{i,j}\geq w(P^{\tau}_{ij})$;
\item[{\rm (ii)}] there exist $i$ and $j$ such that $i\neq j$, $i\taccesses j$,
$j\neq\operatorname{endn}(\tau,k)$ and $a_{j,i}\geq (w(P^{\tau}_{ij}))^{-1}$.
\end{itemize}
\end{theorem}

In the case of $\cTag^{=1}(A)$, we have to replace condition (i) by a more elaborate one.

\begin{theorem}
\label{t:criter-var}
Let $\tau\in\cTag^{=1}(A)$ and $k=o(\tau)$.
Then $x^{(\tau,k)}$ is not an extremal if and only if one of the following conditions hold:
\begin{itemize}
\item[{\rm (\u{\i})}] there exist $i,l$ and $j$ such that $i\neq l\neq j$
and either $\tau(i)\neq o_{c\tau}$ and $a_{i,j}\geq w(P_{ij}^{\tau})$ 
or $\tau(i)=o_{c\tau}$ and $a_{i,j}>w(P_{ij}^{\tau})$.
\item[{\rm (ii)}] there exist $i$ and $j$ such that $i\neq j$, $i\taccesses j$,
$j\neq\operatorname{endn}(\tau,k)$ and $a_{j,i}\geq (w(P^{\tau}_{ij}))^{-1}$.
\end{itemize}
\end{theorem}

\begin{proof}[Proof of Theorem~\ref{t:criter} and Theorem~\ref{t:criter-var}] 
Without loss of generality we will assume that
the nodes of $\tau$, where $\tau$ is a cycle or a germ, are numbered
in such a way that $k=1$ and $\tau(i)=i+1$ for all $i\in\supp(\tau)$ except
for the node $\operatorname{endn}(\tau,k)$ which has the greatest number
$m$. Note that if $\tau$ is a cycle then $\tau(m)=1$, and otherwise $\tau(m)=o_{c\tau}$.
With such numbering, conditions (i), (ii) and (\u{\i}) take the following form:
\begin{itemize}
\item[{\rm (i')}] there exist $i,j$ such that $i+1<j$
and $a_{i,j}\geq w(P^{\tau}_{ij})$;
\item[{\rm (ii')}] there exist $i,j$ such that $i<j$,
$j\neq m$ and $a_{j,i}\geq (w(P^{\tau}_{ij}))^{-1}$.
\item[{\rm (\u{\i}')}] there exist $i,j$ such that $i<j$,
and either $o_{c\tau}\neq i+1$ and $a_{i,j}\geq w(P_{ij}^{\tau})$ 
or $o_{c\tau}=i+1$ and $a_{i,j}>w(P_{ij}^{\tau})$
\end{itemize}

The ``only if'' part: Suppose that $x^{(\tau,1)}$ is not an extremal.
As $V^*(A)=\spann(S)$ where $S$ is defined in~\eqref{e:kitai},
by Proposition~\ref{p:multiorder} and Theorem~\ref{t:kitai}, there exists $\tau'$ and $s$ such that $x^{(\tau',s)}\leq_1 x^{(\tau,1)}$ and
$x^{(\tau',s)}\neq x^{(\tau,1)}$.  As $x^{(\tau',s)}\leq_1 x^{(\tau,1)}$, it follows that
$\supp(\tau')\subseteq\supp(\tau)$ 
and that $1\in\supp(\tau')$ so that $s\longrightarrow_{\tau'} 1$. 
Also since $x^{(\tau',1)}\leq_1 x^{(\tau',s)}$ by Corollary~\ref{c:xtau},
and since $\leq_1$ is transitive, we can assume $s=1$.

Now suppose there exist $i,j$ such that $i+1<j$ and $j=\tau'(i)$. Consider the
least such $i$ and $j$. Condition $x^{(\tau',1)}\leq_1 x^{(\tau,1)}$ means that
$x_l^{(\tau',1)}\leq x_l^{(\tau,1)}$ for all $l\in\supp(\tau)$. In terms of
walks, this means that $w(P_{1l}^{\tau'})^{-1}\leq w(P_{1l}^{\tau})^{-1}$, or
equivalently, $w(P_{1l}^{\tau'})\geq w(P_{1l}^{\tau})$  for all  
$l\in\supp(\tau)$. In particular, this implies $a_{i,j}\geq w(P_{ij}^{\tau})$,
thus we have (i').

Suppose that there are no such $i,j$. Then it can be verified that we have
$\tau'(s)=s+1$ for all $s\in\supp(\tau')$ except for one node $j$ for which
$i=\tau'(j)<j$. However, if $j=m$ then $x^{(\tau',1)}=x^{(\tau,1)}$, a
contradiction. Hence $j<m$, and the edge $(j,i)$ belongs to the unique cycle of $\tau'$.
The other edges of that cycle form the walk $P^{\tau}_{ij}$ and the cycle is
in~$\cycles^{\geq 1}(A)$, hence we have (ii').

It remains to prove that if $\tau\in\cTag^{=1}(A)$ and not (i') or (ii'), then we have (\u{\i}').
So suppose that condition (ii') does not hold, $\tau\in\cTag^{=1}(A)$ and there exist only $i$ and $j$
with $i<j$, $o_{c\tau}=i+1$ and, by contradiction, that $a_{i,j}=w(P_{ij}^{\tau})$ for all 
such $i$ and $j$. Then we have $x^{(\tau',1)}\leq_1 x^{(\tau,1)}$ only for $\tau'=\tau$ (trivially),
or for $\tau'$ such that $\supp(\tau')=\supp(\tau)$,
$\tau'(i)=j$ for some selection of $j$
and $\tau'(k)=\tau(k)$ for all $k\in\supp(\tau)\backslash\{i\}$.
However, it can be checked that
$x^{(\tau',1)}=x^{(\tau,1)}$ for all such $\tau'$ since $\suppp(x^{(\tau',1)})=\suppp(x^{(\tau,1)})$ and the
weight of the unique cycle of $\tau$ is $1$. This implies that there are no vectors preceding 
$x^{(\tau,1)}$ with respect to $\leq_1$ and different from $x^{(\tau,1)}$, a contradiction. Hence we have (\u{\i}').

The ``if'' part: By Proposition~\ref{p:multiorder} and Lemma~\ref{l:xtau}, it is enough to
show that there exists $\tau'\in\cTag(A)\cup\cycles^{\geq 1}(A)$ such that
$x^{(\tau',1)}\leq_1 x^{(\tau,1)}$.

Suppose that (i') or (\u{\i}') holds, and take any such $i$ and $j$.
Denote by $c$ the (unique) cycle of $\tau$.
Define $\tau'$ by
\begin{equation}
\label{e:taupdef}
\begin{split}
\supp(\tau')&=
\begin{cases}
\{1,\ldots,i\}\cup\{j,\ldots,m\}, &\text{if $o_{c\tau}\leq i$ or $o_{c\tau}\geq j$},\\
\{1,\ldots,i\}\cup\{o_{c\tau},\ldots,m\}, &\text{if $i<o_{c\tau}<j$}.
\end{cases}\\
\tau'(l)&=
\begin{cases}
\tau(l), &\text{if $l\in\supp(\tau')$, $l\neq i$,}\\
j, &\text{if $l=i$.}
\end{cases}
\end{split}
\end{equation}
The definition of $\tau'$ and the inequality $a_{i,j}\geq w(P_{ij}^{\tau})$ immediately imply
$w(P_{1l}^{\tau'})\geq w(P_{1l}^{\tau})$ for all $l\in\{1,\ldots,i\}\cup\{j,\ldots,m\}$. For the
case when $l\in\{o_{c\tau},\ldots,j\}$ (if $i<o_{c\tau}<j$), observe that
$w(P_{1l}^{\tau'})\geq w(P_{1l}^{\tau})\cdot w(c)\geq w(P_{1l}^{\tau})$.
Thus $w(P_{1l}^{\tau'})\geq w(P_{1l}^{\tau})$ holds for all $l\in\supp(\tau')$,
implying the inequalities $x_l^{(\tau',1)} (x_1^{(\tau',1)})^{-1}\leq
x_l^{(\tau,1)} (x_1^{(\tau,1)})^{-1}$ for all $l\in\supp(\tau')$. Hence
$x^{(\tau',1)}\leq_1 x^{(\tau,1)}$. It remains to show that 
$x^{(\tau',1)}\neq x^{(\tau,1)}$.

Observe that $\supp(\tau')$ is a proper subset of $\supp(\tau)$ unless
when $o_{c\tau}=i+1$
(that is, the cycle begins at the next node after $i$). 
If $\supp(\tau')$ is a proper subset of $\supp(\tau)$ then clearly
$x^{(\tau',1)}\neq x^{(\tau,1)}$. If $o_{c\tau}=i+1$,
we verify that
for all $l\in\{o_{c\tau},\ldots,j\}$, we have that
either $w(P_{1l}^{\tau'}> w(P_{1l}^{\tau})\cdot w(c) \geq w(P_{1l}^{\tau})$ 
(if $a_{i,j}>w(P_{ij}^{\tau})$) or
$w(P_{1l}^{\tau'}\geq w(P_{1l}^{\tau})\cdot w(c) > w(P_{1l}^{\tau})$
(if $w(c)>1$), and then also $x^{(\tau',1)}\neq x^{\tau}$.


If (i') or (\u{\i}') do not hold but (ii') does, then define $\tau'$ by
\begin{equation}
\supp(\tau')=\{1,\ldots,j\},\quad
\tau'(l)=
\begin{cases}
i, &\text{if $l=j$},\\
\tau(l)=l+1, &\text{if $l<j$}.
\end{cases}
\end{equation}
Then the condition $a_{j,i}\geq w(P_{ij}^{\tau})^{-1}$ implies that
$P_{ij}^{\tau}$ and $(j,i)$ constitute a nonnegative cycle, hence $\tau'\in\cTag(A)\cup\cycles^{\geq 1}(A)$
and $\supp(\tau')$ is a proper
subset of $\supp(\tau)$. Thus we have $x^{(\tau',1)}\leq_1 x^{(\tau,1)}$.
\end{proof}

It remains to consider the case when $\tau$ is a cycle with weight $1$.
The set of such cycles is denoted by $\cycles^{=1}(A)$.
In this case
all vectors $x^{(\tau,i)}$ are proportional to 
each other, for all $i\in \supp(\tau)$.  Therefore we will denote $x^{\tau}=x^{(\tau,i)}$, where
$i$ is an arbitrary index of $\supp(\tau)$.

\begin{theorem}
\label{t:unitycase}
Let  $\tau\in\cycles^{=\bunity}(A)$ and $i\in\supp(\tau)$.
Then $x^{\tau}$ is not an extremal if and only if there
exist two edges $(k_1,l_1)$ and $(k_2,l_2)$ such that
$k_1,l_1,k_2,l_2\in\supp(\tau)$,
$l_1\notin\tau(k_1)$, $l_2\notin\tau(k_2)$,
$k_1\neq k_2$, $a_{k_1,l_1}\cdot w(P^{\tau}_{l_1k_1})\geq 1$
and $a_{k_2,l_2}\cdot w(P^{\tau}_{l_2k_2})\geq 1$.
\end{theorem}
\begin{proof} Let $x^{\tau}$ be not an extremal, then for each $i\in \supp(\tau)$
there exist $\tau_i$ and $i'$ such that $x^{(\tau_i,i')}\leq_i x^{(\tau,i)}$, and hence  
$(k_i,l_i)$ with $k_i,l_i\in\supp(\tau)$ and $a_{k_il_i}\cdot w(P^{\tau}_{l_ik_i})\geq 1$. Indeed, if there is no such edge then the 
domain of any cycle or germ other than $\tau$ includes a node not in $\supp(\tau)$, while all generators 
derived from $\tau$ are proportional to $x^{\tau}$.
Furthermore, some
$k_i$'s should be different, at least for two values of $i$. Indeed, if all $k_i$ are equal to the same index denoted by $k$, then we have 
$x^{(\tau_i,\tau(k))}=x^{\tau}$ for all $i$, while $\tau(k)$ does not belong to the support of any other
vector derived from the germ $\tau_i$, for any $i$.

For the converse implication, let $(k_1,l_1)$ and $(k_2,l_2)$ be the two edges satisfying given conditions, and let $\tau_1$ and
$\tau_2$ be defined by
\begin{equation}
\tau_1(i)=
\begin{cases}
\tau(i), & \text{if $i\in \supp(\tau)\backslash\{k_1\}$},\\
l_1, & \text{if $i=k_1$}.
\end{cases}, \quad
\tau_2(i)=
\begin{cases}
\tau(i), & \text{if $i\in \supp(\tau)\backslash\{k_2\}$},\\
l_2, & \text{if $i=k_2$}.
\end{cases}
\end{equation}

Since $k_1\neq k_2$, for each $i\in \supp(\tau)$, 
either $i\neq \tau(k_1)$ or $i\neq\tau(k_2)$,
and we define $\tau':=\tau_1$ or $\tau':=\tau_2$ respectively. Then we have
$x^{(\tau',i)}\leq_i x^{\tau}$ and $x^{(\tau',i)}\neq x^{\tau}$. As such a
vector can be found for any $i$, $x^{\tau}$ is not extremal.

\end{proof}

\section{Acknowledgement}
The author is grateful to Professor Peter Butkovi\v{c} for useful discussions and advice.

\end{document}